\documentclass[12pt]{article}
\usepackage{amssymb,amsmath,amsthm,psfrag,graphics,latexsym,graphicx}

\numberwithin{equation}{section}
\newtheorem{theorem}{Theorem}[section]
\newtheorem{lemma}[theorem]{Lemma}

\newtheorem{conjecture}[theorem]{Conjecture}
\newtheorem{claim}[theorem]{Claim}

\newtheorem{remark}[theorem]{Remark}

\newtheorem{construction}[theorem]{Construction}

\title{Counterexamples to the List Square Coloring Conjecture}

\author{
\begin{tabular}{c}
{\sc Seog-Jin KIM\thanks{This work was supported by the National Research Foundation of Korea(NRF) grant funded by the Korea government (MEST) (No. 2011-0009729).}} \\
[1ex]
{\small
Department of Mathematics Education, Konkuk University,
Seoul 143-701, Korea} \\
{\small
{\it E-mail address}: {\tt skim12@konkuk.ac.kr}} \\
\\
{\sc Boram PARK\thanks{Corresponding author: borampark22@gmail.com}}\\
[1ex]
{\small National Institute for Mathematical Sciences, Daejeon 306-390, Korea} \\
{\small
{\it E-mail address}: {\tt borampark22@gmail.com}} \\
\end{tabular}
}

\date{May 10, 2013}

\begin{document}

\maketitle

\begin{abstract}
The square $G^2$ of  a graph $G$ is the graph defined on $V(G)$ such that two vertices $u$ and $v$ are adjacent in $G^2$ if the distance between $u$ and $v$ in $G$ is at most 2.  Let $\chi(H)$ and $\chi_l(H)$ be the chromatic number and the list chromatic number of $H$, respectively.  A graph $H$ is called {\em chromatic-choosable} if $\chi_l (H) = \chi(H)$.
It is an interesting problem to find graphs that are
 chromatic-choosable.
Kostochka and Woodall \cite{KW2001} conjectured that $\chi_l(G^2) = \chi(G^2)$ for every graph $G$, which is called  List Square Coloring Conjecture.
In this paper, we give infinitely many counterexamples to the conjecture.
Moreover, we show  that the value $\chi_l(G^2) - \chi(G^2)$ can be arbitrary large.
\end{abstract}




\section{Introduction}

A proper $k$-coloring $\phi: V(G) \rightarrow \{1, 2, \ldots, k \}$ of a graph $G$ is an assignment of colors to the vertices of $G$ so that any two adjacent vertices  receive distinct colors.
The {\em chromatic number} $\chi(G)$ of a graph $G$ is the least $k$ such that there exists a proper $k$-coloring of $G$.
A list assignment $L$ is an assignment of lists of colors to vertices.
A graph $G$ is said to be {\em $k$-choosable} if for any list $L(v)$ of size at least $k$, there exists a proper coloring $\phi$ such that $\phi(v) \in L(v)$ for every $v \in V(G)$.
The least $k$ such that $G$ is $k$-choosable is called the {\it list chromatic number} $\chi_\ell(G)$ of a graph $G$. Clearly $\chi_l(G) \geq \chi(G)$ for every graph $G$.

A graph $G$ is called {\em chromatic-choosable} if $\chi_l (G) = \chi(G)$. It is an interesting problem to determine which graphs are chromatic-choosable.  There are several famous conjectures that some  classes of graphs
are chromatic-choosable including the List Coloring Conjecture.

Given a graph $G$, the {\em total graph}  $T(G)$ of $G$  is the graph such that
$V(T(G)) = V(G) \cup E(G)$, and two vertices $x$ and $y$ are adjacent in $T(G)$ if
(1) $x,y\in V(G)$,  $x$ and $y$ are adjacent vertices in $G$, or
(2) $x,y\in E(G)$,  $x$ and $y$ are adjacent edges in $G$, or
(3) $x\in V(G)$, $y\in E(G)$, and $x$ is incident to $y$ in $G$.
The {\it line graph} $L(G)$ of a graph $G$ is the graph such that $V(L(G))=E(G)$ and two vertices $x$ and $y$ are adjacent in $L(G)$ if and only if $x$ and $y$ are adjacent edges in $G$.

The famous List Coloring Conjecture (or called Edge List Coloring Conjecture) is stated as follows, which was proposed independently by Vizing, by Gupa, by Albertson and Collins, and by Bollob\'{a}s and Harris (see \cite{Toft} for detail).

 \begin{conjecture}\label{LECC}
{\rm \bf (List  Coloring Conjecture)} For any graph $G$, $\chi_l(L(G)) = \chi(L(G))$.
\end{conjecture}

It was shown that the List  Coloring Conjecture is true for some graph families, see~\cite{G1995, PW1999, W1999}.
On the other hand, Borodin, Kostochka, Woodall \cite{BKW1997} proposed the following conjecture as a version of the famous List Coloring Conjecture for total graphs.

\begin{conjecture}\label{LTCC}
{\rm \bf (List Total Coloring Conjecture)} For any graph $G$, $\chi_l(T(G)) = \chi(T(G))$.
\end{conjecture}

For a simple graph $G$, the {\it square} $G^2$ of $G$ is defined such that $V(G^2) = V(G)$ and two vertices $x$ and $y$ are adjacent in $G^2$ if and only if the distance between $x$ and $y$ in $G$ is at most 2.
Kostochka and Woodall \cite{KW2001} proposed the following conjecture.

\begin{conjecture} \label{LSCC}
{\rm \bf (List Square Coloring Conjecture)} For any graph $G$, $\chi_l(G^2)=\chi(G^2)$.
\end{conjecture}

Note that the List Square Coloring Conjecture implies the List Total Coloring Conjecture.
If $H$ is the graph obtained by placing a vertex in the middle of every edge of a graph $G$, then
$H^2 = T(G)$.  Hence if the List Square Coloring Conjecture is true for a special class of bipartite graphs, then the List Total Coloring Conjecture is true.

\medskip

The List Square Coloring Conjecture has attracted a lot of  attention and been cited in many papers related with coloring problems so far, and it has been widely accepted to be true.
The List Square Coloring Conjecture has been proved for several small classes of graphs.

\medskip
In this paper, we disprove the List Square Coloring Conjecture by showing that there exists a graph $G$  such that
$\chi_l (G^2) \neq \chi(G^2)$.
We show that for each prime $n \geq 3$,
there exists a graph $G$ such that $G^2$ is the complete multipartite graph $K_{n*(2n-1)}$, where
$K_{n*(2n-1)}$ denotes the complete multipartite graph with $(2n-1)$ partite sets in which each partite set has size $n$.
Note that  $\chi_l ( K_{n*(2n-1)}) >  \chi( K_{n*(2n-1)})$ for every integer $n \geq 3$. Thus there exist infinitely many counterexamples to the List Square Coloring Conjecture.
Moreover,
we show that the gap between $\chi_l (G^2)$  and $\chi(G^2)$ can be arbitrary large, using
the property that $\chi_l ( K_{n*(2n-1)}) -  \chi( K_{n*(2n-1)}) \geq n-1$ for every integer $n \geq 3$.

\medskip
In the next section, first we construct a graph $G$, and next we will show that $G^2$ is a complete multipartite graph by proving several lemmas.

\section{Construction}

Let $[n]$ denote $\{1,2,\ldots,n\}$. A \textit{Latin square} of order $n$ is an $n\times n$ array such that in each cell, an element of $[n]$ is arranged and  there is no same element in each row and each column.
For a Latin square $L$ of order $n$, the element on the $i$th row and the $j$th column is denoted by $L(i,j)$.
For example, $L$ in Figure \ref{Latin-square}  is a Latin square of order 3, and  $L(1,2)=2$, $L(1,3)=3$, and $L(3,2)=3$.
Two Latin squares $L_1$ and $L_2$ are \textit{orthogonal} if for any $(i,j) \in [n]\times [n]$, there exists unique $(k,\ell)\in[n]\times [n]$ such that $L_1(k,\ell)=i$ and $L_2(k,\ell)=j$. For example, $L_1$ and $L_2$ in Figure \ref{Latin-square}  are orthogonal.
\begin{figure}[b!]
\vspace{0.5cm}
\[
L={\footnotesize\begin{tabular}{|c|c|c|}
                  \hline
                  1 & 2 & 3 \\ \hline
                  3 & 1 & 2 \\ \hline
                  2 & 3 & 1 \\
                  \hline
                \end{tabular}}  \hspace{1cm}
L_1={\footnotesize\begin{tabular}{|c|c|c|}
                  \hline
                  1 & 2 & 3 \\ \hline
                  2 & 3 & 1 \\ \hline
                  3 & 1 & 2 \\
                  \hline
                \end{tabular}} \hspace{1cm}
                L_2={\footnotesize\begin{tabular}{|c|c|c|}
                  \hline
                  1 & 3 & 2 \\ \hline
                  2 & 1 & 3 \\ \hline
                  3 & 2 & 1 \\
                  \hline
                \end{tabular}}
\]
\caption{Latin squares $L, L_1$ and $L_2$ of order $3$}
\label{Latin-square}
\end{figure}

From now on, we fix a prime number $n$ with $n\ge 3$ in this section.
For $i\in [n]$,  we define a Latin square $L_i$ of order $n$ by
\begin{eqnarray} \label{eq:Latin}
 L_i(j,k)= j+i(k-1) \pmod{n}, \quad \mbox{ for } (j, k) \in [n] \times [n].
\end{eqnarray}
Then it is  (also well-known) easily checked that $L_i$ is a Latin square of order $n$ and $\{L_1,L_2,\ldots, L_{n-1}\}$ is a  family of mutually orthogonal Latin squares of order $n$ as $n$ is prime (see page 252 in \cite{Van-Lint}).
For example, in  Figure~\ref{Latin},  $L_1$, $L_2$, $L_3$, and $L_4$ are Latin squares of order 5 when $n = 5$.
\begin{figure}
\[
L_1= {\footnotesize \begin{tabular}{|c|c|c|c|c|}
 \hline
  1 & 2 & 3 & 4 & 5 \\ \hline
  2 & 3 & 4 & 5 & 1 \\ \hline
  3 & 4 & 5 & 1 & 2 \\ \hline
  4 & 5 & 1 & 2 & 3 \\ \hline
  5 & 1 & 2 & 3 & 4 \\ \hline
\end{tabular}}
 \quad \quad  L_2={\footnotesize \begin{tabular}{|c|c|c|c|c|}
  \hline
  1 & 3 & 5 & 2 & 4 \\ \hline
  2 & 4 & 1 & 3 & 5 \\ \hline
  3 & 5 & 2 & 4 & 1 \\ \hline
  4 & 1 & 3 & 5 & 2 \\ \hline
  5 & 2 & 4 & 1 & 3 \\ \hline
\end{tabular}}  \]
\[  L_3={\footnotesize\begin{tabular}{|c|c|c|c|c|}
  \hline
  1 & 4 & 2 & 5 & 3 \\ \hline
  2 & 5 & 3 & 1 & 4 \\ \hline
  3 & 1 & 4 & 2 & 5 \\ \hline
  4 & 2 & 5 & 3 & 1 \\ \hline
  5 & 3 & 1 & 4 & 2 \\ \hline
\end{tabular}}
\quad \quad  L_4={\footnotesize\begin{tabular}{|c|c|c|c|c|}
  \hline
  1 & 5 & 4 & 3 & 2 \\ \hline
  2 & 1 & 5 & 4 & 3 \\ \hline
  3 & 2 & 1 & 5 & 4 \\ \hline
  4 & 3 & 2 & 1 & 5 \\ \hline
  5 & 4 & 3 & 2 & 1 \\ \hline
\end{tabular}}
\]
\caption{$\{L_1,L_2,L_3,L_4\}$ is a family of mutually orthogonal Latin squares of order $n$ defined in (\ref{eq:Latin}).}
\label{Latin}
\end{figure}

\medskip

Now we will construct a graph $G$ which is a counterexample to Conjecture~\ref{LSCC}.

\begin{construction}\label{construction} \rm
For each prime number $n \geq 3$, we construct a graph $G$ with $2n^2-n$ vertices as follows.
For $1 \leq i \leq n$, let $P_i$ be the set of $n$ elements such that
\begin{eqnarray*}
P_i&=&\{ v_{i,1}, v_{i,2}, ..., v_{i,n}\}
\end{eqnarray*}
and for $1 \leq j \leq n-1$, let $Q_j$ be the set of $n$ elements such that
\begin{eqnarray*}
Q_j&=&\{ w_{j,1}, w_{j,2}, ...., w_{j,n}\}.
\end{eqnarray*}
Let $\{L_1,L_2,\ldots, L_{n-1}\}$ is the family of mutually orthogonal Latin squares of order $n$
obtained by (\ref{eq:Latin}).  Graph $G$ is defined as follows.
\begin{itemize}
\item[] $V(G) = \left( \cup_{i =1}^{n} P_i \right) \bigcup \ \left( \cup_{j =1}^{n-1} Q_j \right) = P_1\cup  \cdots \cup P_n \cup Q_1\cup \cdots \cup Q_{n-1}$.
\item[] $E(G) = E_1 \cup E_2$ such that
\begin{eqnarray*}
E_1&=&\bigcup_{i\in [n-1]}\bigcup_{j\in [n]} \{ w_{i,j}  v_{k,L_i(j,k)} : 1 \leq k \leq n \}, \\
E_2&=& \bigcup_{j\in [n]} \{ xy : x,y\in T_j\},
\end{eqnarray*}
where
\[T_j=\{ v_{1,j}, v_{2,j}, \ldots, v_{n,j}\}  \quad \mbox{for }1 \leq j \leq n.\]
\end{itemize}
\end{construction}

\medskip

That is, for each $i\in [n]$, $T_i$ is a clique of size $n$ in $G$, and  $T_1$, $T_2$, \ldots, $T_n$ are mutually vertex disjoint.
And for  $i\in [n-1]$ and  for  $j\in [n]$,
\begin{eqnarray}\label{Neighbor}
N_G(w_{i,j})&=&\{ v_{1,L_i(j,1)} ,v_{2,L_i(j,2)}, ...,  v_{n,L_i(j,n)} \},
\end{eqnarray}
which is  obtained  by reading the $j$th row of the Latin square $L_i$ defined in~(\ref{eq:Latin}).
See Figure~\ref{fig1} for an illustration of the case when $n=3$.

\begin{figure}[b!]
\begin{center}
\psfrag{a}{\footnotesize$v_{1,1}$}
\psfrag{b}{\footnotesize$v_{1,2}$}
\psfrag{c}{\footnotesize$v_{1,3}$}
\psfrag{d}{\footnotesize$v_{2,1}$}
\psfrag{e}{\footnotesize$v_{2,2}$}
\psfrag{f}{\footnotesize$v_{2,3}$}
\psfrag{g}{\footnotesize$v_{3,1}$}
\psfrag{h}{\footnotesize$v_{3,2}$}
\psfrag{i}{\footnotesize$v_{3,3}$}
\psfrag{m}{\footnotesize$w_{1,1}$}
\psfrag{n}{\footnotesize$w_{1,2}$}
\psfrag{o}{\footnotesize$w_{1,3}$}
\psfrag{p}{\footnotesize$w_{2,1}$}
\psfrag{q}{\footnotesize$w_{2,2}$}
\psfrag{r}{\footnotesize$w_{2,3}$}
\psfrag{A}{\footnotesize$P_1$}
\psfrag{B}{\footnotesize$P_2$}
\psfrag{C}{\footnotesize$P_3$}
\psfrag{E}{\footnotesize$Q_1$}
\psfrag{F}{\footnotesize$Q_2$}
\psfrag{j}{$N_G(w_{1,1})=\{ v_{1,\bf{1}},v_{2,\bf{2}},v_{3,\bf{3}}\}$}
\psfrag{k}{$N_G(w_{1,2})=\{ v_{1,\bf{2}},v_{2,\bf{3}},v_{3,\bf{1}}\}$}
\psfrag{l}{$N_G(w_{1,3})=\{ v_{1,\bf{3}},v_{2,\bf{1}},v_{3,\bf{2}}\}$}
\psfrag{u}{$N_G(w_{2,1})=\{ v_{1,\bf{1}},v_{2,\bf{3}},v_{3,\bf{2}}\}$}
\psfrag{t}{$N_G(w_{2,2})=\{ v_{1,\bf{2}},v_{2,\bf{1}},v_{3,\bf{3}}\}$}
\psfrag{s}{$N_G(w_{2,3})=\{ v_{1,\bf{3}},v_{2,\bf{2}},v_{3,\bf{1}}\}$}
 \includegraphics[width=9cm]{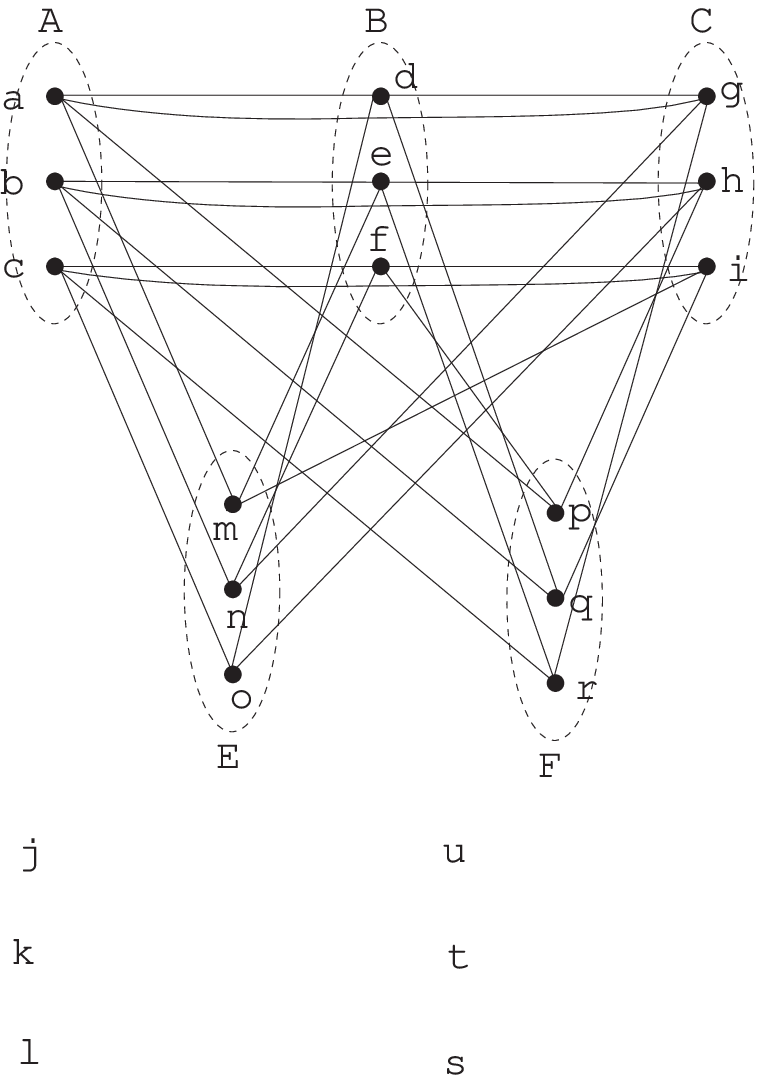}\\
 \[ L_1={\footnotesize\begin{tabular}{|c|c|c|}
                  \hline
                  1 & 2 & 3 \\ \hline
                  2 & 3 & 1 \\ \hline
                  3 & 1 & 2 \\
                  \hline
                \end{tabular} }\qquad
 L_2={\footnotesize\begin{tabular}{|c|c|c|}
                  \hline
                  1 & 3 & 2 \\ \hline
                  2 & 1 & 3 \\ \hline
                  3 & 2 & 1 \\
                  \hline
                \end{tabular}} \]
\caption{The graph $G$ and  Latin squares $L_1$ and $L_2$ defined in (\ref{eq:Latin}) when $n=3$.
In  $N_G(w_{i,j})$, the bold subscripts are the $j$th row of the Latin square $L_i$.
   }\label{fig1}
\end{center}
\end{figure}

\bigskip

From now on, we denote $G$ the graph defined in  Construction \ref{construction}.
We will show that $G^2$ is the complete multipartite graph $K_{n \star (2n-1)}$ whose partite sets are $P_1$,  $\ldots$, $P_n$, $Q_1$,  $\ldots$, $Q_{n-1}$.
For simplicity, let $P=P_1\cup  \cdots \cup P_n$ and $Q=Q_1\cup  \cdots \cup Q_{n-1}$.
From the definition of $G$, the following lemma holds.


\begin{lemma}\label{N(w)}
The graph $G$ satisfies the following properties.
\begin{itemize}
\item[\rm(1)] For every  $x\in Q$,
\[|N_G(x) \cap P_k|=1, \mbox{ for each }  1 \leq k \leq n.\]
\item[\rm(2)] For every $x\in Q$,
\[|N_G(x) \cap T_k|=1, \mbox{ for each }  1 \leq k \leq n.\]
\item[\rm(3)] If $x$ and $y$ are distinct vertices in $Q$, then
\[|N_{G}(x)\cap N_G(y)|\le 1. \]
In particular, if $x,y\in Q_i$ for some $i\in [n-1]$, then
\[|N_{G}(x)\cap N_G(y)|=0. \]
\end{itemize}
\end{lemma}

\begin{proof}
Let $x \in Q$, denoted $x=w_{i,j}$.
By (\ref{Neighbor}), it is clear that $N_G(x)$ contains exactly one vertex $v_{k, L_{i}(j,k)}$ of $P_k$ for each $k\in [n]$.
Therefore (1) is true.

\medskip
Let $x$ be a vertex in $Q$, denoted  $x=w_{i,j}$. For each $1\le k\le n$,
there exists unique $\ell\in [n]$ such that $L_{i}(j,\ell)=k$ since
$L_{i}$ is a Latin square.
Then $v_{\ell, L_{i}(j,\ell)}=v_{\ell,k} \in N_G(w_{i,j})$  by (\ref{Neighbor}).
By the definition of $T_k$, $v_{\ell,k} \in T_k$.
Therefore, $v_{\ell,k} \in N_G(w_{i,j})\cap T_k$. From the uniqueness of $\ell$,
$|N_G(w_{i,j})\cap T_k| = 1$.
Thus (2) is true.

\medskip
Next we will prove (3).  Let $x$ and $y$ be  two vertices in $Q$, denoted $x=w_{i,j}$ and $y=w_{i',j'}$.
By (\ref{Neighbor}),
\begin{eqnarray*}
N_G(w_{i,j})  &=&\{ v_{1,L_i(j,1)} ,v_{2,L_i(j,2)}, ..., v_{n,L_i(j,n)} \}, \\
N_G(w_{i',j'})&=&\{ v_{1,L_{i'}(j',1)} ,v_{2,L_{i'}(j',2)}, ...,v_{n,L_{i'}(j',n)} \}.
\end{eqnarray*}
\begin{claim}\label{claim_nw} It holds that
$v_{k,L_i(j,k)}\in N_{G}(w_{i,j})\cap N_G(w_{i',j'})$ if and only if
\begin{equation} \label{equation}
(i-i')(k-1) \equiv j'- j \pmod{n}.
\end{equation}
\end{claim}
\begin{proof} It is easy to see that
\begin{eqnarray*}
&&v_{k,L_i(j,k)}\in N_{G}(w_{i,j})\cap N_G(w_{i',j'})\\
\Leftrightarrow && v_{k,L_i(j,k)}=v_{k,L_{i'}(j',k)}  \\
\Leftrightarrow && L_i(j,k)=L_{i'}(j',k) \\
\Leftrightarrow && j+i(k-1)\equiv j'+i'(k-1) \pmod{n}\\
\Leftrightarrow && (i-i')(k-1) \equiv j'- j \pmod{n}.
\end{eqnarray*}
\end{proof}

Suppose that
$v_{k,L_i(j,k)}, v_{k',L_{i'}(j',k')} \in N_{G}(w_{i,j})\cap N_G(w_{i',j'})$ for some $k,k'\in [n]$.
Then by Claim~\ref{claim_nw},
\begin{eqnarray*}
&& (i-i')(k-1) \equiv j'- j \pmod{n},\\
&& (i-i')(k'-1) \equiv j'- j \pmod{n}.
\end{eqnarray*}
By subtracting two equations,
\[  (i-i')(k-k')\equiv 0 \pmod{n}.\]
First, consider the case when $i-i'\neq 0$.  Then $k-k'\equiv 0 {\pmod{n}}$ since $n$ is prime.
Since $1\le k ,k' \le n$, we have $k=k'$.
Consequently
$\{v_{k,L_i(j,k)}, v_{k',L_{i'}(j',k')} \} = \{v_{k,L_i(j,k)}, v_{k,L_{i'}(j',k)} \}$.
Note that $\{v_{k,L_i(j,k)}, v_{k,L_{i'}(j',k)} \} \subset T_k$.
Thus $N_{G}(w_{i,j})\cap N_G(w_{i',j'}) \subset T_k$.
Therefore  by (2), we have
\[ |N_{G}(w_{i,j})\cap N_G(w_{i',j'})| =|N_{G}(w_{i,j})\cap N_G(w_{i',j'}) \cap T_k|\leq  |N_G(w_{i,j}) \cap T_k| \le  1.\]

Next, consider the case when $i = i'$.
Suppose that $N_{G}(w_{i,j})\cap N_G(w_{i',j'})\neq \emptyset$.
Then (\ref{equation}) is true for some $k$.
Since $i=i'$, (\ref{equation}) is equivalent to $j\equiv j' \pmod{n}$.
Since $1 \leq j, j' \leq n$, we have $j = j'$.
It implies that $w_{i,j}= w_{i',j'}$,  which is a contradiction for the assumption that $w_{i,j}\neq w_{i',j'}$.
Therefore $N_{G}(w_{i,j})\cap N_G(w_{i',j'}) = \emptyset$.
Thus (3) is true.
\end{proof}

\begin{lemma}\label{N(v)}
The graph $G$ satisfies the following properties.
\begin{itemize}
\item[\rm (1)]  If $x\in P$,
\[|N_G(x) \cap Q_k|=1, \mbox{ for each }  1 \leq k \leq n-1.\]
\item[\rm (2)]  If $x$ and $y$ are distinct vertices in $P$, then
\[|N_{G}(x)\cap N_G(y)\cap Q|\le 1. \]
\end{itemize}
\end{lemma}
\begin{proof}
Let $x\in P$.
Suppose that $|N_{G}(x) \cap Q_k|\ge 2$ for some $k \in \{1, \ldots, n-1\}$.
Then  there are two vertices $y,z\in Q_k$ such that $x \in N_{G}(y)\cap N_G(z)$, which implies that $N_{G}(y)\cap N_G(z)\neq \emptyset$.
By (3) of Lemma~\ref{N(w)}, it is impossible. Thus (1) is true.

\medskip
Let $x$ and $y$ be distinct vertices in $P$.
Suppose that $|N_{G}(x)\cap N_G(y)\cap Q|\ge 2$.
Then there exist two vertices $z_1,z_2\in Q$ such that $z_1,z_2 \in N_{G}(x)\cap N_G(y)\cap Q$.
Then $x,y\in N_G(z_1)\cap N_G(z_2)$, and consequently $|N_G(z_1)\cap N_G(z_2)|\ge  2$.
It is a contradiction to (3) of Lemma~\ref{N(w)}.
Thus (2) is true.
\end{proof}


\begin{lemma}\label{independent}
For each $1\le i\le n$, $P_i$ is an independent set of $G^2$.
Also for each $1\le i\le n-1$, $Q_i$ is an independent set of $G^2$.
\end{lemma}

\begin{proof}
For $1 \leq i \leq n$, let $v_{i,j}$ and $v_{i,j'}$ be any two vertices in $P_i$.
Suppose that $v_{i,j}$ and $v_{i,j'}$ are adjacent in $G^2$.
Then there exists a common neighbor $x$ of $v_{i,j}$ and $v_{i,j'}$ since $v_{i,j}$ and $v_{i,j'}$ are not adjacent in $G$.
It follows that $x \in Q$ by Construction~\ref{construction}.
Thus  $v_{i,j}, v_{i,j'}\in N_G(x)\cap P_i$, and so $|N_G({x})\cap P_i|\ge 2$.
It is a contradiction to (1) of Lemma~\ref{N(w)}.
Therefore, $P_i$ is an independent set in $G^2$.

Next we will show that $Q_i$ is an independent set in $G^2$. For $1 \leq i \leq n-1$, let  $w_{i,j}$ and $w_{i,j'}$ be distinct vertices in $Q_i$.
Suppose that $w_{i,j}$ and $w_{i,j'}$ are adjacent in $G^2$.
Then there exists a common neighbor $y$ of $w_{i,j}$ and $w_{i,j'}$ since $w_{i,j}$ and $w_{i,j'}$ are not adjacent in $G$.
It follows that $y \in P$ by the construction of $G$.
Thus  $w_{i,j}, w_{i,j'}\in N_G(y)\cap Q_i$, and so $|N_G({y})\cap Q_i|\ge 2$.
It is a contradiction to (1) of Lemma~\ref{N(v)}.
Therefore, $Q_i$ is an independent set in $G^2$.
\end{proof}

\begin{lemma}\label{adjacent_all}
For any vertex $x\in P$ and for any vertex $y\in Q$, $x$ and $y$ are adjacent in $G^2$.
\end{lemma}

\begin{proof}
Let $x$ and $y$ be vertices in $P$ and $Q$, respectively.
Since $P$ is the disjoint union of $T_1, \ldots, T_n$ which are defined in Construction \ref{construction},
 $x\in T_k$ for some $k \in \{1, \ldots, n \}$.
By (2) of Lemma~\ref{N(w)}, $N_G(y)\cap T_k \neq \emptyset$ and $G[T_k]$ induces a complete subgraph in $G$. Therefore the distance between $x$ and $y$ in $G$ is at most 2.
Thus $x$ is adjacent to $y$ in $G^2$.
\end{proof}


Now we will show that the subgraphs induced by $P$ and $Q$ in $G^2$, denoted $G^2[P]$ and $G^2[Q]$, respectively, are  complete multipartite graphs.
Let $K_{n*r}$ denote the complete multipartite graph with $r$ partite sets in which each partite set has size $n$.

\begin{lemma}\label{thm:G[P]}
The subgraph induced by $P$ in $G^2$, denoted $G^2[P]$,  is $K_{n*n}$ whose partite sets are $P_1$, $P_2$, \ldots, $P_n$.
\end{lemma}

\begin{proof}
Let \[\mathcal{F}=\{ N_G(w) :  w\in Q \}\cup \{ T_1,T_2,\ldots, T_n\}.\]
Note that for each $w\in Q$,
the subgraph induced by $N_G(w)$ in $G^2$ is a complete graph and $N_G(w) \subset P$.
And each $T_i$ is a clique in $G^2$ and $T_i \subset P$ by the definition of $T_i$.
Therefore, $\mathcal{F}$ is a family of cliques in $G^2[P]$.

We will show that
for any $X,Y\in \mathcal{F}$, we have $|X\cap Y|\le 1$.
By (3) of Lemma~\ref{N(w)}, for any two vertices $x, y\in Q$, $|N_G(x)\cap N_G(y)|\le 1$.
By the definition, $|T_i\cap T_j|=0$.
Also, by (2) of Lemma~\ref{N(w)}, $|N_G(w)\cap T_j|= 1$ for any $w$ in $Q$ and for any $1 \leq j \leq n$.
Thus for any $X,Y\in \mathcal{F}$, we have $|X\cap Y|\le 1$,
which implies  that any two cliques of $\mathcal{F}$ are edge-disjoint.

Note that $|\mathcal{F}|=|Q|+n=n(n-1)+n=n^2$.
Thus $\mathcal{F}$ is a family of $n^2$ mutually edge-disjoint cliques in
$G^2[P]$. 
As each clique of $\mathcal{F}$ is $K_n$ and $K_n$ has ${n\choose 2}$ edges, we have
\[|E(G^2[P])| \ge n^2\times {n\choose 2} =n^2 \times \frac{n(n-1)}{2}.\]
On the other hand, by Lemma~\ref{independent},
$E(G^2[P])$ has at most ${n^2 \choose 2} - n\times {n\choose 2}$ edges, since each of $P_1,P_2,\ldots,P_n$ is an independent set in $G^2$.
Note that
\[{n^2 \choose 2} - n\times {n\choose 2}=\frac{n^2(n^2-1)}{2}-\frac{n^2(n-1)}{2} =  n^2 \times \frac{n(n-1)}{2}.\]
Thus \[|E(G^2[P])| = n^2 \times \frac{n(n-1)}{2}.\]
This implies that  $G^2[P]$  is  $K_{n*n}$ whose partite sets are $P_1$, $P_2$, \ldots, $P_n$.
\end{proof}


The following lemma holds by a similar  argument with Lemma~\ref{thm:G[P]}.

\begin{lemma}\label{thm:G[Q]}
 The subgraph induced by $Q$ in $G^2$, denoted $G^2[Q]$,  is  $K_{n*(n-1)}$ whose partite sets are
 $Q_1, \ldots, Q_{n-1}$.
\end{lemma}

\begin{proof}
Let \[\mathcal{F}=\{ N_G(v)\cap Q \mid  v\in P \}.\]
For each $v\in P$, the subgraph induced by $N_G(v)\cap Q$ in $G^2$ is a complete graph.
Therefore, $\mathcal{F}$ is a family of cliques in $G^2[Q]$.
By (2) of Lemma~\ref{N(v)}, for any two vertices $x, y\in P$, we have $|N_G(x)\cap N_G(y)\cap Q|\le 1$.
Therefore any two cliques of $\mathcal{F}$ is edge-disjoint.

Note that $|\mathcal{F}|=|P|=n^2$.
Thus $\mathcal{F}$ is a family of $n^2$ mutually edge-disjoint cliques in
$G^2[Q]$. 
As each clique of $\mathcal{F}$ is  $K_{n-1}$ and  $K_{n-1}$ has ${n-1\choose 2}$ edges, we have
\[|E(G^2[Q])| \ge n^2 \times {n-1\choose 2} =n^2\times\frac{(n-1)(n-2)}{2}.\]
On the other hand, by Lemma~\ref{independent},
$E(G^2[Q])$ has at most ${n^2-n \choose 2} - (n-1)\times {n\choose 2}$ edges, since each of $Q_1$,\ldots, $Q_{n-1}$ is an independent set in $G^2$.
Note that
\[ {n^2-n \choose 2} - (n-1)\times {n\choose 2} =\frac{(n^2-n)(n^2-n-1)}{2} -\frac{(n-1)n(n-1)}{2} =\frac{n^2(n-1)(n-2)}{2} .\]
Thus \[|E(G^2[Q])| = n^2 \times \frac{(n-1)(n-2)}{2}.\]
This implies that  $G^2[Q]$  is  $K_{n*(n-1)}$ whose partite sets are $Q_1$,\ldots, $Q_{n-1}$.
\end{proof}


Now by Lemmas~\ref{independent}, ~\ref{adjacent_all}, ~\ref{thm:G[P]} and~\ref{thm:G[Q]}, we conclude that
the square $G^2$ of $G$ is the complete multipartite graph $K_{n*(2n-1)}$  whose partite sets are $P_1$, $P_2$, $\ldots$, $P_n$, $Q_1$, $\ldots$, $Q_{n-1}$,  which implies   the following main theorem.
%

\begin{theorem}\label{G2}
For each prime $n \geq 3$,
there exists a graph $G$ such that $G^2$ is the complete multipartite graph $K_{n*(2n-1)}$.
\end{theorem}

The following lower bound on the list chromatic number of a complete multipartite graph was obtained in \cite{Vetrik2012}.

\begin{theorem}\label{thm:Vetrik}{\rm (Theorem 4, \cite{Vetrik2012})}
For a complete multipartite graph $K_{n*r}$ with $n,r \ge 2$,
\[\chi_\ell (K_{n*r}) > (n-1)\lfloor\frac{2r-1}{n} \rfloor.\]
\end{theorem}
\begin{proof}
The proof 
is the same as in \cite{Vetrik2012}.  We include it here for the convenience of readers.

Let $A_1, \ldots, A_n$ be a family of disjoint color sets such that $||A_i| - |A_j|| \leq 1$ for each $1 \leq i,j \leq n$ and $|\bigcup_{j=1}^n A_j| = 2r -1$.  Then $|A_j| \geq \lfloor \frac{2r-1}{n} \rfloor$ for each $j \in \{1, \ldots, n\}$.

Define a list assignment $L$ as follows. For $1 \leq i \leq  r$, let $V_i = \{v_{i1}, \ldots, v_{in} \}$ be
the $i$th partite set in $K_{n \star r}$.  For each $v_{ik} \in V_i$, define $L(v_{ik}) = \bigcup_{j=1}^n A_j \setminus A_k$.  Then $|L(x)| \geq (n-1)\lfloor \frac{2r-1}{n} \rfloor$ for each vertex $x$ in $K_{n \star r}$.

Note that in any coloring from these lists, at least two colors on each partite $V_i$ are used.
Thus at least $2r$ colors are needed to have a proper coloring from the lists, but $|\bigcup_{j=1}^n A_j| = 2r -1$.  Hence $K_{n \star r}$ is not $L$-colorable.  This implies that $\chi_\ell (K_{n*r}) > (n-1)\lfloor\frac{2r-1}{n} \rfloor$.
\end{proof}

Consequently, we obtain that  $\chi_{\ell}(G) >\chi(G)$ by the following theorem.

\begin{theorem} \label{main-theorem}
For each prime $n \geq 3$, if $G$ is the graph defined in Construction~\ref{construction}, then
\[ \chi_{\ell}(G^2) - \chi(G^2)  \geq  n-1. \]
\end{theorem}

\begin{proof}
It is clear that $\chi(G^2)=2n-1$  by Theorem~\ref{G2}.
On the other hand, by Theorems~\ref{G2} and~\ref{thm:Vetrik},
\begin{eqnarray*}
\chi_\ell(G^2)&=&\chi_{\ell} (K_{n*(2n-1)}) > (n-1)\lfloor\frac{4n-3}{n} \rfloor \geq 3(n-1),
\end{eqnarray*}
when $n\ge 3$.
Thus for $n \geq 3$,
\[ \chi_\ell(G^2) - \chi(G^2) \geq n-1.\]
\end{proof}

\begin{remark} \rm
Since there are infinitely many primes, from Theorem~\ref{main-theorem},
the gap $\chi_l(G^2) - \chi(G^2)$ can be arbitrary large.
\end{remark}

%
%


\end{document}